\documentclass[a4paper,11pt,reqno,twoside]{amsart}

\usepackage{amsmath}
\usepackage{amsfonts}
\usepackage{amssymb}
\usepackage{amsthm}
\usepackage{color}
\usepackage{ifpdf}
\usepackage{array}
\usepackage{url}
\usepackage{multirow,bigdelim}
\usepackage{ascmac}

\numberwithin{equation}{section}

\newtheorem{theorem}{Theorem}[section]
\newtheorem{lemma}{Lemma}[section]
\newtheorem{proposition}{Proposition}[section]

\newcommand*{\C}{\mathbb{C}}
\newcommand*{\R}{\mathbb{R}}
\newcommand*{\Z}{\mathbb{Z}}

\newcommand{\comment}[1]{}
\title[Li coefficients as norms of functions in a model space]%
      {Li coefficients as norms of functions\\  in a model space} 
\author[M. Suzuki]{Masatoshi Suzuki}
\date{Version of \today}
\AtBeginDocument{%
\begin{abstract}
It is known that the nonnegativity of Li coefficients 
is a necessary and sufficient condition for the Riemann hypothesis. 
We show that it is a necessary and sufficient condition for the Riemann hypothesis 
that all Li coefficients are norms of certain concrete 
functions on the real line. 
Such conditional formulas for Li coefficients are understood as a kind of Weil's criterion 
for the Riemann hypothesis. 
\end{abstract}
\maketitle
}
\begin{document}

%
\section{Introduction and the result} 
%

Li's criterion \cite[Theorem 1]{Li97} obtained by Xian-Jin Li 
asserts that the Riemann hypothesis holds 
if and only if the Li coefficients defined by 
\begin{equation} \label{s101}
\lambda_n := 
\sum_{\rho} \left[1 - \left(1 - \frac{1}{\rho} \right)^n\right]
\end{equation}
are nonnegative for all positive integers $n$, 
where $\rho$ runs over all nontrivial zeros of the Riemann zeta function $\zeta(s)$ 
counting multiplicity 
and the sum is understood as $\sum_\rho = \lim_{T \to \infty}\sum_{|\Im(\rho)| \leq T}$. 
The nonnegativity of $\lambda_n$ follows easily from the Riemann hypothesis 
which states that all nontrivial zeros $\rho$ satisfy $\Re(\rho)=1/2$, 
since the set of all nontrivial zeros is closed under complex conjugation and 
 the vertical line $\Re(s)=1/2$ is mapped to the unit circle by 
$s \mapsto 1-1/s$. 
\medskip

E. Bombieri and J. C. Lagarias \cite{BoLa99} pointed out that 
the nonnegativity of $\lambda_n$ follows from the nonnegativity of 
the Weil distribution $W$ defined by 
\begin{equation} \label{eq0301_2}
f ~\mapsto ~ W(f)=\sum_{\rho} \int_{0}^{\infty} f(x) \,x^{\rho-1} \, dx
\end{equation}
for smooth and compactly supported functions $f$, 
where $\sum_\rho$ has the same meaning as above. 
More precisely, they proved that 
\begin{equation} \label{s102}
2\lambda_n = W(g_n(x) \ast \overline{x^{-1}g_n(x^{-1})})
\end{equation}
holds for  
\begin{equation} \label{eq0225_2}
g_n(x) := 
\begin{cases} 
~\displaystyle{\sum_{j=1}^{n}\binom{n}{j} \frac{(\log x)^{j-1}}{(j-1)!} } & \text{if $0<x<1$}, \\[10pt]
~n/2 & \text{if $x=1$}, \\ 
~0 &  \text{if $x>1$}, 
\end{cases}
\end{equation}
and  all positive integers $n$, where $\ast$ is the multiplicative convolution 
on the multiplicative group $\R_{>0}$. 
In this case, the functions $g_n(x)$ have no information about the nonnegativity of 
Li coefficients $\lambda_n$. 
The result of this paper provides a contrasting situation 
in which certain functions themselves closely related to the nonnegativity of $\lambda_n$.
\medskip

To state the result, using the Riemann xi-function 
\[
\xi(s) := \frac{1}{2} s(s-1)\pi^{-s/2}\Gamma\left(\frac{s}{2}\right)\zeta(s)
\]
and coefficients of the Laurent expansion 
\begin{equation} \label{s103}
-\frac{\zeta'}{\zeta}(s+1) = \frac{1}{s} + \sum_{k=0}^{\infty} \eta_k s^k, 
\end{equation}
we define the function $H_n(s)$
of the $s$-variable 
for a positive integer $n$ by
\begin{equation} \label{s104}
\aligned 
H_n(s) 
& := \frac{\xi(s)}{\xi(s)+\xi'(s)} \left\{ \frac{1}{s-1}  
+ \left[ 1 - \left(1-\frac{1}{s}\right)^n\right] 
\left(  \frac{\xi'}{\xi}(s)
- \frac{1}{s-1} 
- \frac{\xi'}{\xi}(0)-1\right) \right. \\
& \quad \left. -
\sum_{j=2}^{n} \binom{n}{j} \frac{(-1)^{j-1}}{s^j} 
\sum_{k=1}^{j-1}
\Bigl[ (-1)^k \eta_k+  (1-2^{-k-1})\zeta(k+1) \Bigr] s^k \right\},
\endaligned 
\end{equation}
where the second line on the right-hand side 
is understood to be zero when $n=1$. 
Further, we define the function $G_n(z)$ 
of the $z$-variable by 
\begin{equation} \label{eq0223_1}
G_n(z) := H_n\left(\frac{1}{2}-iz \right). 
\end{equation}
Then, $G_n(z)$ is continuous on the real line ($z \in \R$)  
and belongs to $L^2(\R)$ unconditionally as will be shown in Proposition \ref{prop_202} below.

\begin{theorem} \label{thm_1}
Let $G_n(z)$ be the function on the real line 
defined by \eqref{s104} and \eqref{eq0223_1} for a positive integer $n$. 
A necessary and sufficient condition for the Riemann hypothesis is that 
\begin{equation} \label{s105}
\lambda_n = \frac{1}{2\pi} \Vert G_n \Vert_{L^2(\R)}^2 
\end{equation}
holds for all positive integers $n$, 
where $\Vert \cdot \Vert_{L^2(\R)}$ 
on the right-hand side stands for the norm of $L^2(\R)$. 
\end{theorem}

The Riemann hypothesis follows trivially from equation \eqref{s105} by Li's criterion.  
Therefore, the nontrivial part of Theorem \ref{thm_1} 
is that the validity of equation \eqref{s105} is a necessary condition. 
In contrast, 
the validity of equation \eqref{s102}  is unconditional, 
but the nonnegativity of the right-hand side is highly nontrivial.
\medskip

If we assume the Riemann hypothesis, 
$G_n(z)$ are vectors in a {\it model space} $\mathcal{K}(\Theta)$ in the upper half-plane 
generated by a {\it meromorphic inner function} $\Theta$,  
as shown in the proof of Theorem \ref{thm_1} below. 
The title of this paper derives from this fact.

Although precise definitions and related notions are deferred to Section \ref{section_3_1}, 
we should mention that model spaces are important subjects in analysis. 
They are defined as orthogonal complements of shift-invariant closed subspaces 
of the Hardy space and are relevant to many aspects of functional, complex, and harmonic analysis 
as presented in the survey of Garcia--Ross \cite{GaRo15} (see also the book of 
Garcia--Mashreghi--Ross \cite{GMR}). 
Among the model spaces in the upper-half plane, 
those generated by the meromorphic inner functions 
have extra importance in their relation with the de Branges spaces, 
which are nice Hilbert spaces consisting of entire functions. 
The above meromorphic inner function $\Theta$ has the form 
$\Theta(z)=\overline{E(\bar{z})}/E(z)$ with $E(z)=\xi(1/2-iz)+\xi'(1/2-iz)$ 
and the model space $\mathcal{K}(\Theta)$ 
is isomorphic to the de Branges space $\mathcal{H}(E)$ generated by $E$. 
The space $\mathcal{H}(E)$ was introduced and studied by Lagarias \cite{La06} 
and has an interesting property that one of the self-adjoint extensions 
of the operator of multiplication by the independent variable 
has the zeros of $\xi(1/2-iz)$ as eigenvalues. 
This fact makes us interested in the spectral theoretical meaning of $G_n(z)$ 
as a motivation beyond the importance of the model spaces in spectral theory, 
but unfortunately it is not clear.
\medskip

We may compare definition \eqref{s104} 
with the known formula 
\begin{equation} \label{eq0301_1}
\lambda_n 
= -\sum_{j=1}^{n}\binom{n}{j}\eta_{j-1}
+ 1-(\gamma_0+\log 4\pi)\frac{n}{2}
- \sum_{j=2}^{n}\binom{n}{j}(-1)^{j-1}(1-2^{-j})\zeta(j)
\end{equation}
obtained in \cite[Theorem 2 and (4.1)]{BoLa99}, 
where $\gamma_0$ is the Euler--Mascheroni constant.  
Then we observe that several similar terms appear in parallel in both formulas. 
For example, if $n=1$, 
\begin{equation} \label{eq0301_3}
H_1(s)  = -\frac{1}{s} \left(
\frac{\xi'}{\xi}(0) - \frac{\xi'(s)}{\xi(s)}  \right)\left( 1 + \frac{\xi'(s)}{\xi(s)}  \right)^{-1}\\
\end{equation}
by \eqref{s104} and a little calculation, and 
\begin{equation} \label{eq0301_4}
\lambda_1
= -\frac{\xi'}{\xi}(0) 
= \sum_{\rho} \frac{1}{\rho}
= \frac{1}{2}\gamma_{0}
+ 1-\frac{1}{2}\log 4\pi =  0.0230957+ %
\end{equation}
by \cite[p. 282, Remark]{BoLa99} and $\eta_0=-\gamma_0$ in \cite[p. 286, formula below (4.3)]{BoLa99}. 
The reason for the above observation lies in a direct relation between $G_n(z)$ and $g_n(x)$ 
under the Riemann hypothesis by the framework of \cite{Su23}, 
and equation \eqref{s105} expresses that relation via \eqref{s102}.
Since this is the direct background of Theorem \ref{thm_1}, 
we give some details in Section \ref{section_5_2}. 
Other meanings or characterizations of $G_n(z)$ are interesting 
but unknown  at present.
\medskip

Theorem \ref{thm_1} is proved in Section \ref{section_4} 
using results in Section \ref{section_2} 
and the theory of model spaces reviewed in Section \ref{section_3_1}.  
The strategy of the proof of Theorem \ref{thm_1} 
is similar to \cite{Su22}, 
however the computational details change. 
The core of the proof is 
Proposition \ref{prop_201}, which is unconditional, 
and Proposition \ref{prop_203}, which holds under the Riemann hypothesis.

%
\section{Preliminaries} \label{section_2}
%

\subsection{Two unconditional propositions} 

We denote by $\mathcal{Z}$ the set of all distinct zeros $\rho$ 
of $\xi(s)$ (counting zeros without multiplicity) 
and we let $m_\rho$ count the multiplicity of 
the zero of $\xi(s)$ at $s=\rho$. 
If $\rho \in \mathcal{Z}$, then $1-\rho$ and $\overline{\rho}$ 
also belong to $\mathcal{Z}$ with the same multiplicity 
by two functional equations $\xi(s)=\xi(1-s)$ and $\xi(s)=\overline{\xi(\bar{s})}$. 
Also, $0 < \Re(s) < 1$ for every $\rho \in \mathcal{Z}$. 
The Riemann hypothesis is that all $\rho \in \mathcal{Z}$ satisfy $\Re(\rho)=1/2$.

\begin{lemma} \label{lem_201}
For a positive integer $n$, we define
\begin{equation} \label{s204}
M_n(s) 
:= -i\sum_{\rho \in \mathcal{Z}} 
m_\rho \, \left[\, 1 - \left(1-\frac{1}{\rho} \right)^n \,\right]
\frac{1}{s-\rho}.
\end{equation}
Then, $M_n(s)$ is a meromorphic function on $\C$ 
such that all poles are simple and 
$\mathcal{Z}$ is the set of all poles.
\end{lemma}
\begin{proof}
The series on the right-hand side of \eqref{s204} 
converges absolutely and uniformly 
on every compact subset of 
$\C\setminus\mathcal{Z}$, 
since $\sum_{\rho \in \mathcal{Z}}m_\rho|\rho|^{-1-\delta}<\infty$ for any $\delta>0$ 
(because $\xi(s)$ is order one).  
Hence, we obtain the desired conclusion. 
\end{proof}

\begin{proposition} \label{prop_201}
Let $H_n(s)$ and $M_n(s)$ 
be functions defined 
in \eqref{s104} and \eqref{s204}, respectively, 
for a positive integer $n$. 
Then, it holds 
\begin{equation} \label{s205}
H_n(s)
=\frac{i\xi(s)}{\xi(s)+\xi'(s)}\,M_n(s)
\end{equation}
for $s \in \C$. In addition, $\rho \in \mathcal{Z}$ are removable poles of $H_n(s)$.  
\end{proposition}
\begin{proof} 
The main tool for the proof is Weil's explicit formula 
\begin{equation*} 
\aligned 
\lim_{T \to \infty} & \sum_{{\rho \in \mathcal{Z}}\atop{|\Im(\rho)|\leq T}} m_\rho
\int_{0}^{\infty} f(x) \, x^{\rho-1} \, dx \\
& = \int_{0}^{\infty} f(x) \, dx + \int_{0}^{\infty} x^{-1} f(x^{-1}) \, dx 
 - \sum_{m=2}^{\infty} \Lambda(m)(f(m)+m^{-1}f(m^{-1})) \\
& \quad 
- (\log 4\pi + \gamma_0) f(1)
- \int_{1}^{\infty} \left\{ f(x) + x^{-1} f(x^{-1})-2x^{-1} f(1)\right\} \frac{x\,dx}{x^2-1}
\endaligned 
\end{equation*}
in \cite[p. 186]{Bo01} with the conditions for test function $f(x)$ 
in \cite[Section 3]{BoLa99}. 
(Note that the formula in \cite{BoLa99} has two typographical errors 
in the second line of the right-hand side.) 
For a positive integer $n$ and $s \in \C$, 
we define 
\[
\aligned 
f_{s,n}(x):& = 
\begin{cases}
~\displaystyle{-i \sum_{j=1}^{n} \binom{n}{j} \frac{(-1)^{j-1}}{s^j} 
 \sum_{k=0}^{j-1}\frac{(-s\log x)^k}{k!}}, & 0<x<1,  \\[20pt]
~\displaystyle{ - ix^{-s} \left[ 1 - \left(1-\frac{1}{s}\right)^n\right] }, &x \geq 1. 
\end{cases}
\endaligned
\] 
We have  
\[
\int_{0}^{\infty} g_n(x) \, x^{\rho-1} \,dx 
= \left[ 1- \left(1-\frac{1}{\rho} \right)^n \right] 
\]
when $\Re(\rho)>0$ by \cite[Lemma 2]{BoLa99} and
\[
\int_{1}^{\infty} (-ix^{-s})\cdot x^{\rho-1} \,dx = \frac{-i}{s-\rho} 
\]
when $\Re(s)>\Re(\rho)$. 
The multiplicative convolution of $g_n(x)$ and $(-ix^{-s})\mathbf{1}_{(1,\infty)}(x)$ is 
$f_{s,n}(x)$ as 
\[
\aligned 
\int_{0}^{\min(1,x)} &  (-i(x/y)^{-s})g_n(y)\, \frac{dy}{y} 
 =-ix^{-s} \sum_{j=1}^{n} \binom{n}{j} \frac{1}{(j-1)!} 
\int_{0}^{x}  y^s(\log y)^{j-1} \, \frac{dy}{y}  \\
& = -i x^{-s}\sum_{j=1}^{n} \binom{n}{j} \frac{1}{(j-1)!} \cdot \frac{x^{s}(-1)^{j-1}(j-1)!}{s^j} 
 \sum_{k=0}^{j-1}\frac{(-s\log x)^k}{k!} 
\endaligned 
\]
for $0<x<1$ and 
\[
\aligned 
\int_{0}^{\min(1,x)} & (-i(x/y)^{-s})g_n(y)\, \frac{dy}{y} 
 =-ix^{-s} \sum_{j=1}^{n} \binom{n}{j} \frac{1}{(j-1)!} 
\int_{0}^{1}  y^s(\log y)^{j-1} \, \frac{dy}{y}  \\
& = -i x^{-s}\sum_{j=1}^{n} \binom{n}{j} \frac{1}{(j-1)!} \cdot \frac{(-1)^{j-1}(j-1)!}{s^j}  = -i x^{-s} \left[ 1 - \left(1-\frac{1}{s}\right)^n\right] 
\endaligned 
\]
for $x>1$. 
Therefore, 
\[
\int_{0}^{\infty} f_{s,n}(x)\,x^{\rho} \, \frac{dx}{x} 
= -i \left[ 1- \left(1-\frac{1}{\rho} \right)^n \right] \frac{1}{s-\rho}
\quad \text{when $\Re(s)>\Re(\rho)>0$} , 
\]
and thus the left-hand side of Weil's explicit formula for $f_{s,n}(x)$ 
gives $M_n(s)$ of \eqref{s204} when $\Re(s)>1$. 
Hence, if it is shown that the right-hand side of Weil's explicit formula  for $f_{s,n}(x)$ 
is equal to $-i(1+\xi'(s)/\xi(s))H_n(s)$ when $\Re(s)>1$, 
then \eqref{s205} holds for $s \in \C \setminus \mathcal{Z}$ by analytic continuation. 
\medskip

However $f_{s,n}(x)$ does not satisfy the applicability condition for the growth 
(cf. \cite[Section 3]{BoLa99}), 
so we cannot apply the explicit formula directly. 
Hence, for $0<\epsilon<1$, we replace $f_{s,n}(x)$ 
by its truncation 
\[
f_{s,n,\epsilon}(x)
=
\begin{cases}
~f_{s,n}(x) & \text{if $x>\epsilon$}, \\
~f_{s,n}(\epsilon)/2 & \text{if $x=\epsilon$}, \\
~0 & \text{if $x<\epsilon$}. 
\end{cases}
\]

First, we confirm
\begin{equation} \label{s206}
\lim_{\epsilon \to 0+} \sum_{\rho \in \mathcal{Z}} m_\rho
\int_{0}^{\infty} f_{s,n,\epsilon}(x) \,x^\rho\, \frac{dx}{x} 
=\sum_{\rho \in \mathcal{Z}} m_\rho
\int_{0}^{\infty} f_{s,n}(x) \,x^\rho\, \frac{dx}{x} .
\end{equation}
We have
\[
\aligned 
\int_{1}^{\infty} f_{s,n,\epsilon}(x)\, x^{\rho-1} \, dx
& = - i \left[ 1 - \left(1-\frac{1}{s}\right)^n\right] \frac{1}{s-\rho} 
\endaligned 
\]
and 
\[
\int_{\epsilon}^{1} (\log x)^k x^{\rho-1} \, dx
=
\frac{(-1)^k}{\rho^{k+1}}\left(
k! - \epsilon^\rho
\sum_{l=0}^{k} \frac{k!}{l!} (-\rho \log \epsilon)^l
\right). 
\]
by direct and simple calculations. Therefore, 
\[
\aligned 
\int_{\epsilon}^{1} f_{s,n,\epsilon}(x)\, x^{\rho-1} \, dx
& =
-i \sum_{j=1}^{n} \binom{n}{j} \frac{(-1)^{j-1}}{s^j} 
 \sum_{k=0}^{j-1}
\frac{s^k}{\rho^{k+1}}\\
& \quad 
+i \sum_{j=1}^{n} \binom{n}{j} \frac{(-1)^{j-1}}{s^j} 
 \sum_{k=0}^{j-1}
\frac{s^k \epsilon^\rho}{\rho^{k+1}}
\sum_{l=0}^{k} \frac{1}{l!} (-\rho \log \epsilon)^l \\
& = -i 
\left\{ \left[ 1 - \left(1-\frac{1}{s}\right)^n\right] - \left[ 1 - \left(1-\frac{1}{\rho}\right)^n\right] \right\} \frac{1}{s-\rho} \\
& \quad 
+i \sum_{j=1}^{n} \binom{n}{j} \frac{(-1)^{j-1}}{s^j} 
 \sum_{k=0}^{j-1}
\frac{s^k \epsilon^\rho}{\rho^{k+1}}
\sum_{l=0}^{k} \frac{1}{l!} (-\rho \log \epsilon)^l.
\endaligned 
\]
Hence, assuming $\Re(s)>1$, 
we have 
\[
\aligned 
\int_{0}^{\infty} & f_{s,n}(x) \, x^{\rho-1} \, dx 
- \int_{0}^{\infty} f_{s,n,\epsilon}(x) \, x^{\rho-1} \, dx \\ 
& = -i 
\left[ 1 - \left(1-\frac{1}{\rho}\right)^n\right] \frac{1}{s-\rho} 
-i \frac{\epsilon^\rho}{\rho}\left[ 1 - \left(1-\frac{1}{s}\right)^n\right]  \\
& \quad 
+O\left( \frac{\epsilon^{\Re(\rho)}}{|\rho|^2} \sum_{j=1}^{n} \binom{n}{j} 
 \sum_{k=1}^{j-1}
\frac{1}{|s|^{j-k}}
\sum_{l=0}^{k} \frac{1}{l!} \frac{(\log(1/\epsilon))^l}{|\rho|^{k-1-l}} 
\right)
\endaligned 
\]
for $0<\Re(\rho)<1$ and $|\rho| \geq 1$. 
Thus \eqref{s206} holds if it is shown that 
$\sum_\rho \epsilon^{\rho}/\rho$ and 
$\sum_\rho \epsilon^{\Re(\rho)}/|\rho|^2$ tend to zero, as $\epsilon \to 0+$, 
faster than any nonnegative power of $\log(1/\epsilon)$. 
However, they are established in \cite[pp. 284--285]{BoLa99}.

We return to the calculation of the right-hand side of Weil's explicit formula. 
It is easy to verify
\begin{equation} \label{s207}
\int_{-\infty}^{\infty} f_{s,n}(x) \, dx = 
- \frac{i}{s-1}, 
\end{equation}
\begin{equation} \label{eq0224_1}
\sum_{m=1}^{\infty}\Lambda(m)f_{s,n}(m)=i \left[ 1- \left(1-\frac{1}{s} \right)^n \right] \frac{\zeta'}{\zeta}(s)
\end{equation}
for $\Re(s)>1$ and $m \in \Z_{>0}$, 
\[
\int_{0}^{\infty} f_{s,n,\epsilon}(x) \, \frac{dx}{x}
=  -i 
\sum_{j=1}^{n} \binom{n}{j} \frac{(-1)^{j-1}}{s^j} \frac{1}{s}
\left(1 +  \sum_{k=0}^{j-1}\frac{(s\log(1/\epsilon))^{k+1}}{(k+1)!} \right),  
\]
and 
\[
\aligned 
\sum_{m=1}^{\infty} \frac{\Lambda(m)}{m}\phi_{s,n,\epsilon}(m^{-1}) 
& = -i \sum_{j=1}^{n} \binom{n}{j} \frac{(-1)^{j-1}}{s^j} 
\sum_{1 \leq m \leq 1/\epsilon} \frac{\Lambda(m)}{m} \sum_{k=0}^{j-1}\frac{(s\log m)^k}{k!} 
\endaligned 
\]
by direct calculation. Using the formula
\[
\eta_k = \frac{(-1)^k}{k!} \lim_{\epsilon \to 0+} 
\left[ \sum_{1 \leq m \leq 1/\epsilon} \frac{\Lambda(m)}{m} 
(\log m)^k 
- \frac{\log(1/\epsilon)^{k+1}}{k+1} 
\right]
\]
in \cite[(4.1)]{BoLa99} for coefficients of \eqref{s103}, 
\[
\aligned 
\lim_{\epsilon \to 0+}  & 
\left[
\sum_{1 \leq m \leq 1/\epsilon} \frac{\Lambda(m)}{m} 
\sum_{k=0}^{j-1}\frac{(s\log m)^k}{k!} 
-
\frac{1}{s}
\left(1 +  \sum_{k=0}^{j-1}\frac{(s\log(1/\epsilon))^{k+1}}{(k+1)!} \right)
\right]  \\
& = \sum_{k=0}^{j-1}\frac{s^k}{k!} 
\lim_{\epsilon \to 0+} 
\left[ \sum_{1 \leq m \leq 1/\epsilon} \frac{\Lambda(m)}{m} 
(\log m)^k 
- \frac{\log(1/\epsilon)^{k+1}}{k+1} 
\right]  -\frac{1}{s} \\
& = \sum_{k=0}^{j-1} \eta_k (-s)^k -\frac{1}{s} . 
\endaligned 
\]
Therefore, we obatin
\begin{equation} \label{s208}
\aligned 
\lim_{\epsilon \to 0+}  &
\left[
\int_{0}^{\infty} f_{s,n,\epsilon}(x) \, \frac{dx}{x}
-
\sum_{m=1}^{\infty} \frac{\Lambda(m)}{m} f_{s,n,\epsilon}(m^{-1}) 
\right] \\
& \qquad \qquad  = i \sum_{j=1}^{n} \binom{n}{j} \frac{(-1)^{j-1}}{s^j} 
\left( \sum_{k=0}^{j-1} \eta_k (-s)^k -\frac{1}{s} 
\right) 
\endaligned 
\end{equation}
for arbitrary $s$. 

We turn to the calculation of the fifth term of the right-hand side 
of Weil's explicit formula. 
We have
\[
\frac{x}{x^2-1} = \sum_{m=0}^{M} x^{-2m-1} +O(x^{-2M-3})
\]
\[
\aligned 
\int_{1}^{\infty}  \left\{ f_{s,n}(x) + x^{-1} f_{s,n}(x^{-1})-2x^{-1} f_{s,n}(1)\right\} 
& \sum_{m=0}^{M} x^{-2m-1} \, dx \\
& 
=: I_1(M)+I_2(M)+I_3(M), \quad \text{say}. 
\endaligned 
\]
Then,
\[
I_1(M)
 = - i\left[ 1 - \left(1-\frac{1}{s}\right)^n\right]  \sum_{m=0}^{M} \frac{1}{s+2m}, 
\quad 
I_3(M)
=  i \left[ 1 - \left(1-\frac{1}{s}\right)^n\right] \sum_{m=0}^{M} \frac{2}{2m+1} ,
\]
and 
\[
\aligned 
I_2(M)
& = -i\sum_{j=1}^{n} \binom{n}{j} \frac{(-1)^{j-1}}{s^j} 
 \sum_{k=0}^{j-1} s^k \sum_{m=0}^{M} \frac{1}{(2m+1)^{k+1}} \\
& = -i\sum_{j=2}^{n} \binom{n}{j} \frac{(-1)^{j-1}}{s^j} 
 \sum_{k=1}^{j-1} s^k \sum_{m=0}^{M} \frac{1}{(2m+1)^{k+1}} \\
& \quad -i
 \sum_{m=0}^{M} \frac{1}{(2m+1)} \left[ 1 - \left(1-\frac{1}{s}\right)^n\right]. 
\endaligned 
\]
Therefore, 
\[
\aligned 
\lim_{M \to \infty} 
(I_1(M)+I_2(M)+I_3(M)) 
& = - i\left[ 1 - \left(1-\frac{1}{s}\right)^n\right]  \sum_{m=0}^{\infty} 
\left(\frac{1}{s+2m} - \frac{1}{2m+1} \right) \\
& \quad -i\sum_{j=2}^{n} \binom{n}{j} \frac{(-1)^{j-1}}{s^j} 
 \sum_{k=1}^{j-1} s^k \sum_{m=0}^{\infty} \frac{1}{(2m+1)^{k+1}} \\
& = i\left[ 1 - \left(1-\frac{1}{s}\right)^n\right]  \sum_{m=0}^{\infty} 
\frac{1}{2}\left[ \psi\left(\frac{s}{2}\right) - \psi\left(\frac{1}{2}\right)  \right] \\
& \quad -i\sum_{j=2}^{n} \binom{n}{j} \frac{(-1)^{j-1}}{s^j} 
 \sum_{k=1}^{j-1} s^k  (1-2^{-k-1})\zeta(k+1), 
\endaligned 
\]
where we used the well-known series expansion 
\begin{equation*} 
\psi(w) =
\frac{\Gamma'}{\Gamma}(w) = -\gamma_0 - \sum_{n=0}^{\infty}
\left( \frac{1}{w+n} - \frac{1}{n+1} \right)
\end{equation*}
of the digamma function.  
On the other hand, we have 
\[
\lim_{M \to \infty} 
\int_{1}^{\infty}|f_{s,n}(x) + x^{-1} f_{s,n}(x^{-1})-2x^{-1} f_{s,n}(1)|\, x^{-2M-3}\,dx
= 0
\]
by $f_{s,n}(x) + x^{-1} f_{s,n}(x^{-1})-2x^{-1} f_{s,n}(1) \ll x^{-1}(\log x)^n$ on $[1,\infty)$. 
As a result, 
\begin{equation} \label{s209}
\aligned 
\int_{1}^{\infty} & \left\{ f_{s,n}(x) + x^{-1} f_{s,n}(x^{-1})-2x^{-1} f_{s,n}(1)\right\} 
\frac{x\,dx}{x^2-1} \\
& =  i\left[ 1 - \left(1-\frac{1}{s}\right)^n\right]  \sum_{m=0}^{\infty} 
\frac{1}{2}\left[ \psi\left(\frac{s}{2}\right) - \psi\left(\frac{1}{2}\right)  \right] \\
& \quad -i\sum_{j=2}^{n} \binom{n}{j} \frac{(-1)^{j-1}}{s^j} 
 \sum_{k=1}^{j-1} s^k  (1-2^{-k-1})\zeta(k+1).
\endaligned 
\end{equation}

Combining \eqref{s207}, \eqref{eq0224_1}, \eqref{s208}, and \eqref{s209}, 
the right-hand side of Weil's explicit formula 
for $f_{s,n,\epsilon}$ tends to 
\begin{equation} \label{s211}
\aligned 
\, & - \frac{i}{s-1}
- i \left[ 1- \left(1-\frac{1}{s} \right)^n \right] \frac{\zeta'}{\zeta}(s) 
 + i \sum_{j=1}^{n} \binom{n}{j} \frac{(-1)^{j-1}}{s^j} 
\left( \sum_{k=0}^{j-1} \eta_k (-s)^k -\frac{1}{s} 
\right)   \\
& \quad 
- i\left[ 1 - \left(1-\frac{1}{s}\right)^n\right]  \sum_{m=0}^{\infty} 
\left[ \frac{1}{2}\psi\left(\frac{s}{2}\right) - \frac{1}{2}\psi\left(\frac{1}{2}\right) - \gamma_0 - \log 4\pi\right] \\
& \quad 
 +i\sum_{j=2}^{n} \binom{n}{j} \frac{(-1)^{j-1}}{s^j} 
 \sum_{k=1}^{j-1} s^k  (1-2^{-k-1})\zeta(k+1)
\endaligned 
\end{equation}
as $\epsilon \to 0+$. 
We have $\eta_0=\lim_{s \to 0}( s^{-1}-(\zeta'/\zeta)(1-s))$ by \eqref{s103}. Therefore, 
\begin{equation} 
\eta_0 - \frac{1}{2}\psi\left(\frac{1}{2}\right) + \frac{1}{2}\log \pi 
= -\lim_{s \to 0} \frac{\xi'}{\xi}(1-s) +1
= \frac{\xi'}{\xi}(0) + 1 
\end{equation}
by the functional equation for $\xi(s)$. Therefore, \eqref {s211} equals to 
\begin{equation} \label{s210}
\aligned
\, & -\frac{i}{s-1} - i\left[ 1 - \left(1-\frac{1}{s}\right)^n\right] 
\left(  \frac{\xi'}{\xi}(s)
- \frac{1}{s-1} 
-\frac{\xi'}{\xi}(0)-1\right) \\
& \quad + i 
\sum_{j=2}^{n} \binom{n}{j} \frac{(-1)^{j-1}}{s^j} 
\sum_{k=1}^{j-1}
\Bigl[ (-1)^k \eta_k+  (1-2^{-k-1})\zeta(k+1) \Bigr] s^k.
\endaligned 
\end{equation}
Hence, \eqref {s205} holds for $\Re(s)>1$ by \eqref{s210}. 

Finally, we prove that $\rho \in \mathcal{Z}$ 
are removable singularities of $H_n(s)$. 
We have 
\begin{equation} \label{s212}
\frac{\xi(s)}{\xi(s)+\xi'(s)}
= (s-\rho)\left(\frac{1}{m_\rho}+o(1)\right) 
\end{equation}
near $s=\rho$ by using the series expansion 
$\xi(s) = c(m_\rho)(s-\rho)^{m_\rho}(1+o(1))$ ($c(m_\rho)\not=0$).  
Hence, $H_n(s)$ is analytic in a neighborhood of $s=\rho$ 
by \eqref{s204} and \eqref{s205}.  
\end{proof}

\begin{proposition} \label{prop_202}
Let $G_n(z)$ be functions defined in \eqref{eq0223_1}. 
Then, the restriction of $G_n(z)$ to the real line ($z \in \R$) is bounded, 
real-analytic,  
and belongs to $L^2(\R)$ for all positive integers $n$. 
\end{proposition}
\begin{proof} 
By \eqref{eq0223_1}, it is sufficient to 
prove that $H_n(s)$ is bounded, 
real-analytic, and $L^2$-integrable on the line vertical $\Re(s)=1/2$. 
We have 
\[
\frac{i\xi(s)}{\xi(s)+\xi'(s)}
= \frac{i}{2} \left( 1 + \frac{\xi(s)-\xi'(s)}{\xi(s)+\xi'(s)} \right)
= \frac{i}{2} \left( 1 + \frac{\xi(1-s)+\xi'(1-s)}{\xi(s)+\xi'(s)} \right)
\]
by $\xi(s)=\xi(1-s)$. If $z$ is real, $\xi(1/2+iz)=\overline{\xi(1/2-iz)}$ 
and  $\xi'(1/2+iz)=\overline{\xi'(1/2-iz)}$ by $\xi(\bar{s})=\overline{\xi(s)}$. 
Therefore, 
\[
\left| \frac{\xi(1-s)+\xi'(1-s)}{\xi(s)+\xi'(s)} \right| =1 
\]
when $s=1/2-iz$ for $z \in \R$, where   
zeros of $\xi(s)+\xi'(s)$ in the denominator cancel out in the numerator $\xi(1-s)+\xi'(1-s)$, 
even if they exist. 
The meromorphic function $M_n(s)$ of \eqref{s204} has poles of 
order one at $\rho \in \mathcal{Z}$, 
but $H_n(s)$ is holomorphic there, 
as proved in Proposition \ref{prop_201}. 
Hence, $H_n(s)$ is bounded and analytic on the line $\Re(s)=1/2$ 
by \eqref{s204} and \eqref{s205}.  
On the other hand, in the vertical strip $0 \leq \Re(s) \leq 1$, 
we have $(\Gamma'/\Gamma)(s/2) \ll \log |s|$ (well-known) and 
\[
\frac{\zeta'}{\zeta}(s)
= \sum_{|\Im(s)-\Im(\rho)| \leq 1} \frac{1}{s-\rho}+O(\log|\Im(s)|)
\]
by \cite[Theorem 9.6 (A)]{Tit86}. 
In both estimates, implied constants are uniform in $0 \leq \Re(s) \leq 1$. 
The number of zeros $\rho$ satisfying $|\Im(s)-\Im(\rho)| \leq 1$ 
is $O(\log |\Im(s)|)$ counting with multiplicity by \cite[Theorem 9.2]{Tit86}. 
Therefore, 
$H_n(s) \ll |s|^{-1}\log |s|$ as $|s| \to \infty$ by \eqref{s104}.  
Hence $H_n(s)$ is square integrable on $\Re(s)=1/2$. 
\end{proof}

\subsection{Basic analytical properties of $H_n(s)$}  

We enumerate some basic analytic properties of $H_n(s)$ 
for the convenience of subsequent studies 
although not directly necessary for the main results of this paper. 

\begin{proposition} 
Functions $H_n(s)$ defined in \eqref{s104} 
have the following properties. 
\begin{enumerate}
\item[(1)]  $H_n(s)$ 
are meromorphic functions on $\C$ taking real values on the real line. 
\item[(2)]  $H_n(s)$ are neither real nor purely imaginary valued on the line $\Re(s)=1/2$ in general. 
\end{enumerate}

The zeros of $\xi(s)+\xi'(s)$ are classified into those originating 
from multiple zeros of $\xi(s)$ and those not. 
It is this latter category that appears as poles of $H_n(s)$. 
\begin{enumerate}
\item[(3)] $H_n(s)$ are analytic in $\C$ except for points $\lambda$
such that $\xi(\lambda)+\xi'(\lambda)=0$ and $\lambda\not\in \mathcal{Z}$. 
In particular, $s=0$ and $s=1$ are not poles of $H_n(s)$. 
\comment{
Since $M_n(s)$ is analytic outside $\mathcal{Z}$ by definition \eqref{s204}
and the absolute convergence of the sum 
$\sum_{\rho \in \mathcal{Z}}m_\rho|\rho|^{-2}$, we obtain the following: 
}
\item[(4)]  If $\lambda$ is an zero of $\xi(s)+\xi'(s)$ of order $m$ and 
a zero of $M_n(s)$ of order less than $m$, then it is a pole of $H_n(s)$. 
That is, the locations of possible poles of $H_n(s)$ are independent of $n$, 
but whether or not they are actual poles and their order depend on $n$.
\item[(5)]  The set of zeros of $\xi(s)+\xi'(s)$ 
and the set of poles of $H_n(s)$ are closed under complex conjugation. 
\end{enumerate}
\end{proposition}
\begin{proof} 
The first half of (1) follows from Lemma \ref{lem_201} and Proposition \ref{prop_201}, 
and the second half is trivial by definition \eqref{s104}. 
Formula \eqref{eq0301_3} shows that $H_1(s)$ 
is neither real nor purely imaginary valued on the line $\Re(s)=1/2$, 
since $(\xi'/\xi)(s)$ takes non-constant pure imaginary values on $\Re(s)=1/2$. 
Hence, (2) holds.  
From Lemma \ref{lem_201} and Proposition \ref{prop_201}, 
$H_n(s)$ has no poles other than the zeros of $\xi(s)$ and $\xi(s)+\xi'(s)$. 
Furthermore, from the proof of Proposition \ref{prop_202}, 
$H_n(s)$ has no poles at the zeros of $\xi(s)$. 
Hence, (3) holds. Moreover, (4) holds by \eqref{s205}. 
Since $\xi(s)+\xi'(s)$ and $H_n(s)$ are real valued on the real line, 
(5) holds. 
\end{proof}

%
\section{Proof of Theorem \ref{thm_1}} \label{section_4}
%

\subsection{Model spaces and related notions} \label{section_3_1}

To prove Theorem \ref{thm_1}, 
we need the theory of model spaces in the upper half-plane. 
Therefore, we first briefly review the notions related to model spaces in the upper half-plane 
according to 
Garcia--Ross \cite[Section 10]{GaRo15}, 
Makarov--Poltoratski \cite[Section 1.2]{MaPo05}, 
and Havin--Mashreghi \cite[Section 2]{MR2016246}. 
(See also Garcia--Mashreghi--Ross \cite[Sections 3.6.3 and 5.10.4]{GMR}.)
\medskip

Let $\C_+:=\{z \in \C\,|\, \Im(z)>0\}$ be the upper half-plane. 
Let $\mathbb{H}^\infty(\C_+)$ be the space of all bounded analytic functions in $\C_+$. 
An {\it inner function} in $\C_+$ is a function $\Theta \in \mathbb{H}^\infty(\C_+)$ 
such that $\lim_{y \to 0+}|\Theta(x+iy)|=1$ for almost all $x \in \R$ 
with respect to the Lebesgue measure. 
If an inner function $\Theta$ in $\C_+$ extends to a meromorphic function 
in the whole complex plane $\C$, 
it is called a {\it meromorphic inner function} in $\C_+$. 
\medskip

Let $\mathbb{H}^2:=\mathbb{H}^2(\C_+)$ be the Hardy space in $\C_+$, 
which is the space of all analytic functions $F$ in $\C_+$ 
for which $\sup_{y>0} \int_{\R} |F(x+iy)|^2 \, dx < \infty$. 
As usual, we identify $\mathbb{H}^2$ with a closed subspace of $L^2(\R)$ 
as a Hilbert space by taking an almost everywhere
well-defined boundary function 
$F|_\R(x) := \lim_{y \to 0} F(x+iy)$ on $\R$. 
In particular, $\langle F,G \rangle_{\mathbb{H}^2}
= \langle F|_\R,G|_\R\rangle_{L^2(\R)}$, 
where  $\langle \cdot,\cdot \rangle_{L^2(\R)}$ is the the standard inner product of $L^2(\R)$. 
\medskip

If $\Theta$ is an inner function in $\C_+$, 
the set $\Theta \mathbb{H}^2 = \{ \Theta(z)F(z) \, |\, F \in \mathbb{H}^2\}$ 
forms a closed subspace of $\mathbb{H}^2$ 
invariant with respect to multiplication by all exponentials $\exp(i a z)$, $a>0$.  
For an inner function $\Theta$ in $\C_+$, 
a {\it model subspace} $\mathcal{K}(\Theta)$ in $\C_+$ 
is the subspace of $\mathbb{H}^2$ defined  as the orthogonal complement 
of $\Theta \mathbb{H}^2$ in $\mathbb{H}^2$, that is, 
$\mathcal{K}(\Theta):=\mathbb{H}^2 \ominus \Theta \mathbb{H}^2$. 
By the identification of $\mathbb{H}^2$ with a subspace of $L^2(\R)$, 
the model space 
$\mathcal{K}(\Theta)$ is also regarded as a subspace of $L^2(\R)$. 
If $\Theta(z)$ is a meromorphic inner function, 
all members of $\mathcal{K}(\Theta)$ are meromorphic in $\C$. 
\medskip

A function $E(z)$ in the {\it Hermite--Biehler class} 
is an entire function satisfying inequality 
$|E^\sharp(z)| < |E(z)|$ for all $z \in \C_+$. 
This is called {\it structure function} or {\it de Branges function} in \cite{La06}. 
(It is also called the {\it de Branges structure function} in other literature.)  
Every  $E(z)$ in the Hermite--Biehler class 
generates the {\it de Branges space} $\mathcal{H}(E)$, 
consisting of all entire functions $f(z)$ 
such that both $f(z)/E(z)$ and $\overline{f(\bar{z})}/E(z)$ 
belong to the Hardy space $\mathbb{H}^2$. 

If $\Theta(z)$ is a meromorphic inner function in $\C_+$, 
there exists a function $E(z)$ in the Hermite--Biehler class 
such that $\Theta(z)=\overline{E(\bar{z})}/E(z)$. 
Therefore, $|\Theta(z)|<1$ for $z \in \C_+$. 
Further, the model space $\mathcal{K}(\Theta)$ and 
the de Branges space $E(z)$ are isomorphic as a Hilbert space 
by the mapping $F(z) \mapsto E(z)F(z)$. 
In particular, every $F(z) \in \mathcal{K}(\Theta)$ 
is a meromorphic function in $\C$ such that $E(z)F(z)$ is an entire function. 

The function $E(z)$ in the Hermite--Biehler class 
for one meromorphic inner function $\Theta(z)$ 
is not unique because one can add extra zeros on
the real axis to $E(z)$, they will cancel out of $\Theta(z)$. 
However if we impose the condition that $E(z)$ has no zeros on the real axis, 
it is unique up to multiplication by a constant.  

\subsection{Conditions equivalent to the Riemann hypothesis}

We define 
\begin{equation} \label{s201}
E(z):=\xi(1/2-iz)+\xi'(1/2-iz)
\end{equation}
and
\begin{equation} \label{s202}
\Theta(z):=\overline{E(\bar{z})}/E(z). 
\end{equation}
In \eqref{s201}, $\xi'(1/2-iz)$ means the substitution of $s=1/2-iz$ 
into the derivative $\xi'(s)$.

\begin{proposition} 
Let $E(z)$ and $\Theta(z)$ be functions defined in \eqref{s201} and \eqref{s202}, respectively. 
Then the following are equivalent:  
\begin{enumerate}
\item[(1)] The Riemann hypothesis holds; 
\item[(2)] $E(z)$ belongs to the Hermite--Biehler class;  
\item[(3)] $\Theta(z)$ is a meromorphic inner function for $\C_+$. 
%
\end{enumerate}
\end{proposition}
\begin{proof} 
(1) and (2) are equivalent by \cite[Theorem 1]{La06}. 
The implication of (2) to (3) is immediate. 
This is because $\Theta(z)$ has norm one on the real axis, 
is analytic in $\C_+$, and is bounded in the upper half-plane, 
which makes it an inner function.
We prove that (3) implies (2) by contradiction. 
Note that (3) implies $|\Theta(z)|<1$ for all $z \in \C_+$. 
If (2) is false, then (1) is also false, 
so $\xi(1/2-iz_0)=0$ for some $z_0 \in \C_+$. 
By writing $\xi(1/2-iz)=c(z-z_0)^mF(z)$ 
for some $c \not=0$, $m \in \Z_{>0}$, 
and an entire function $F(z)$ with $F(z_0)\not=0$, 
we see $\Theta(z_0)=-1$. 
This is a contradiction, and therefore (3) implies (2). 
Hence, (2) and (3) are equivalent.
\comment{
(3) and (4) are equivalent by definition of model spaces, 
noting that model spaces in \cite{MaPo05} allow any inner function. 
Therefore, 
the fact that the particular $\Theta(z)$ is meromorphic function (unconditionally) 
is mentioned in (4). If $\Theta(z)$ is holomorphic in $\C_+$, 
it becomes an inner function. 
}
\end{proof}

We comment on the direct relation between (1) and (3). 
The function $\Theta(z)$ in \eqref{s202} is unconditionally meromorphic on $\C$ 
and unconditionally takes absolute value one everywhere on the real axis by definition. 
Therefore, it is an meromorphic inner function if it is analytic in $\C_+$. 
The Riemann hypothesis is used to prove it has no poles in the upper half-plane.

\subsection{A special orthonormal basis in a model space} \label{section_2_2}

For the entire function $E(z)$  defined in  \eqref{s201}, we define 
\begin{equation}   \label{s203}
A(z) := (E(z)+\overline{E(\bar{z})})/2. 
\end{equation}
Then $A(z)=\xi(1/2-iz)$, because $\overline{E(\bar{z})}=\xi(1/2-iz)-\xi'(1/2-iz)$ 
by functional equations $\xi(s)=\xi(1-s)$ and $\xi(s)=\overline{\xi(\bar{s})}$. 
Therefore, the Riemann hypothesis is equivalent 
to all the zeros of $A(z)$ lie on the real axis. 

\begin{proposition} \label{prop_203} 
Let $\Theta(z)$ be a function defined in  \eqref{s202}. 
We assume that the Riemann hypothesis holds 
and denote by $\Gamma~(\subset \R)$ the set of all ordinates of distinct zeros 
$\rho=1/2-i\gamma$ of $\xi(s)$.  
We define 
\begin{equation} \label{s213}
F_\gamma(z) := \sqrt{\frac{m_\gamma}{\pi}} \frac{i(1+\Theta(z))}{2(z-\gamma)} 
\end{equation}
for $\gamma \in \Gamma~ (\subset \R)$. 
Then, the family of functions 
$\{F_\gamma(z)\}_{\gamma \in \Gamma}$ 
forms an orthonormal basis of the Hilbert space $\mathcal{K}(\Theta)$. 
\end{proposition}
\begin{proof} 
Let $A(z)$ be the entire function of \eqref{s203}. 
We have 
\[
\frac{A(iy)}{E(iy)} = 
\left(
1  - \frac{\log \pi}{2}
- \frac{8y}{1-4y^2} 
+ \frac{1}{2}\frac{\Gamma'}{\Gamma}\left(\frac{1}{4}+\frac{y}{2}\right)
+ \frac{\zeta'}{\zeta}\left(\frac{1}{2}+y\right)
 \right)^{-1} \gg (\log y)^{-1}
\]
as $y \to +\infty$ by the Stirling formula for the gamma function 
and the absolutely convergent Dirichlet series of $\zeta'/\zeta(\sigma)$ 
for $\sigma>1$. 
This lower bound shows that $A(z)$ does not belong to $\mathcal{H}(E)$ 
by \cite[Proposition 2.1]{Re02}. 
Therefore, the family $\{A(z)/(z-\gamma)\}_{\gamma \in \Gamma}$ 
forms an orthogonal basis of the de Branges space $\mathcal{H}(E)$ 
by \cite[Theorem 22]{dB68}, 
and hence,  $\{i(1+\Theta(z))/(2(z-\gamma))\}_{\gamma \in \Gamma}$ 
forms an orthogonal basis of $\mathcal{K}(\Theta)$. 

Let $\mu_\Theta$ be the positive discrete measure on $\R$ 
supported on the set $\sigma(\Theta):=\{x\in \R\,|\,\Theta(x)=-1\}=\Gamma$ 
and $\mu_\Theta(x)=2\pi/|\Theta'(x)|$. 
Then the restriction map $F \mapsto F|_{\Gamma}$ 
is a unitary operator $\mathcal{K}(\Theta) \to L^2(\mu_\Theta)$ 
(\cite[Theorem 2.1]{MaPo05}). 
Therefore, 
\begin{equation} \label{s214}
\sqrt{\frac{2}{\pi|\Theta'(\gamma)|}} \frac{i(1+\Theta(z))}{2(z-\gamma)}, \quad \gamma \in \Gamma
\end{equation}
forms an orthonormal basis of $\mathcal{K}(\Theta)$. 
We have $\Theta'(\gamma)/2
=\lim_{z \to \gamma}(1+\Theta(z))/(2(z-\gamma))
=-i/m_\gamma
$ by \eqref{s212}, 
where $m_\gamma$ is the multiplicity of 
the zero of $\xi(s)$ at $s=1/2- i\gamma$. 
Hence, \eqref{s214} is nothing but \eqref{s213}. 
\end{proof}

\subsection{Proof of Theorem \ref{thm_1}} 

It suffices to show that \eqref{s105} holds assuming the Riemann hypothesis 
if considering Li's criterion.  
We assume that the Riemann hypothesis holds 
and denote by $\Gamma ~(\subset \R)$ the set of ordinates of distinct zeros 
$\rho=1/2-i\gamma$ of $\xi(s)$.  
We let $m_\gamma$ count the multiplicity of 
the zero of the function $\xi(s)$ at $s=1/2- i\gamma$. 
If $\gamma \in \Gamma$, then $-\gamma$ and $\overline{\gamma}$ 
also belong to $\Gamma$ with the same multiplicity 
by two functional equations $\xi(s)=\xi(1-s)$ and $\xi(s)=\overline{\xi(\bar{s})}$. 

Substituting $s=1/2-iz$ into \eqref{s204} and \eqref{s205}, 
and then using \eqref{s213}, 
\begin{equation} \label{eq0225_1}
G_n(z) 
= \sum_{\gamma \in \Gamma} 
\sqrt{\pi m_\gamma} \, \left[\, 1 - \left(1-\frac{1}{1/2-i\gamma} \right)^n \,\right]
F_\gamma(z),
\end{equation}
since $i\xi(s)/(\xi(s)+\xi'(s))=i(1+\Theta(z))/2$ 
and $\xi(s)=A(z)$ if $s=1/2-iz$. 
By Proposition \ref{prop_202}, $G_n(z)$ unconditionally belongs to $L^2(\R)$, 
and by Proposition \ref{prop_203}, 
$F_\gamma(z)$ on the right-hand side form an orthonormal basis of 
$\mathcal{K}(\Theta)$. 
We verify that \eqref{eq0225_1} is an expansion of $G_n(z)$ 
with respect to the basis $F_\gamma(z)$ 
by checking the square summability of the coefficients on the right-hand side. 
We have 
\begin{equation} \label{eq_0503_1}
\aligned 
\, & \left\vert\, 1 - \left(1-\frac{1}{1/2-i\gamma} \right)^n \,\right\vert^2 \\
& \qquad  = \left[\, 1 - \left(1-\frac{1}{1/2-i\gamma} \right)^n \,\right]
\left[\, 1 - \left(1-\frac{1}{1/2+i\gamma} \right)^n \,\right] \\
& \qquad =
\left[\, 1 - \left(1-\frac{1}{1/2-i\gamma} \right)^n \,\right]
+
\left[\, 1 - \left(1-\frac{1}{1/2+i\gamma} \right)^n \,\right], 
\endaligned 
\end{equation}
since $\left(1-\frac{1}{1/2+i\gamma} \right)=\left(1-\frac{1}{1/2-i\gamma} \right)^{-1}$ 
for $\gamma \in \R$. The right-hand side is estimated as 
\[
2\Re \left[ 1 - \left(1-\frac{1}{1/2-i\gamma} \right)^n \right] 
=
2\Re \left[ \sum_{k=1}^{n} \binom{n}{k} \frac{(-1)^{k+1}}{(1/2-i\gamma)^k} \right] 
\ll \frac{1}{1/4+\gamma^2}
\]
with an implied constant depending only on $n$. 
Therefore, 
\[
\sum_{\gamma \in \Gamma} m_\gamma 
\left\vert\, 1 - \left(1-\frac{1}{1/2-i\gamma} \right)^n \,\right\vert^2 
\ll \sum_{\gamma \in \Gamma} \frac{m_\gamma}{\gamma^2} < \infty .
\]

Applying Proposition \ref{prop_203} 
to calculate the norm of $G_n(z)$ 
by \eqref{eq0225_1}  
and using \eqref{eq_0503_1}, 
\begin{equation} \label{s301}
\aligned 
\Vert G_n \Vert^2
& = \pi \sum_{\gamma \in \Gamma} 
m_\gamma 
\left\{
\left[\, 1 - \left(1-\frac{1}{1/2-i\gamma} \right)^n \,\right] 
+
\left[\, 1 - \left(1-\frac{1}{1/2+i\gamma} \right)^n \,\right] 
\right\},
\endaligned 
\end{equation}
where the right-hand side converges absolutely under the curly brackets. 

On the other hand, by \eqref{s101} and the convention for the sum $\sum_\rho$ 
in the introduction, 
\begin{equation*} 
\aligned 
\lambda_n 
& = \lim_{T \to \infty} \sum_{{\gamma \in \Gamma}\atop{|\gamma| \leq T}} 
m_\gamma \, 
\left[\, 1 - \left(1-\frac{1}{1/2-i\gamma} \right)^n \,\right] \\
& = 
\lim_{T \to \infty} \frac{1}{2} \sum_{{\gamma \in \Gamma}\atop{|\gamma| \leq T}} 
m_\gamma \, \left\{
\left[\, 1 - \left(1-\frac{1}{1/2-i\gamma} \right)^n \,\right] 
+
\left[\, 1 - \left(1-\frac{1}{1/2+i\gamma} \right)^n \,\right] 
\right\}, 
\endaligned 
\end{equation*}
since $\Gamma$ is closed under $\gamma \mapsto -\gamma$. 
As we have already seen, the right-hand side is absolutely convergent under the curly brackets. 
Therefore, 
\begin{equation} \label{s302} 
\lambda_n 
= 
\frac{1}{2} \sum_{\gamma \in \Gamma} 
m_\gamma \, \left\{
\left[\, 1 - \left(1-\frac{1}{1/2-i\gamma} \right)^n \,\right] 
+
\left[\, 1 - \left(1-\frac{1}{1/2+i\gamma} \right)^n \,\right] 
\right\}.
\end{equation} 
Hence equality \eqref{s105} follows from \eqref{s301} and \eqref{s302}. 
\hfill $\Box$

\section{Concluding remarks} \label{section_5}

\subsection{On a relation between $\lambda_n$ and $G_n$} \label{section_5_1}

It is an interesting question what can be said about the relation 
between Li coefficients $\lambda_n$ and functions $G_n$ 
either unconditionally or when the Riemann hypothesis is false. 
However, it is difficult to say anything concrete 
about the relation between $\lambda_n$ and $G_n$ in such a situation. 
For example, equation 
\begin{equation} \label{eq0301_5}
\lambda_1 = \frac{1}{2\pi} \Vert G_1 \Vert_{L^2(\R)}^2, \quad G_1(z)=H_1(1/2-iz)
\end{equation}
for $n=1$ is a necessary condition for the truth of the Riemann hypothesis. 
In this case, $\lambda_1$ and $H_1(s)$ can be written down in simple form 
as \eqref{eq0301_4} and \eqref{eq0301_3}, respectively, 
so \eqref{eq0301_5} can in principle be confirmed with arbitrary precision 
by numerical calculations. 
However, the author could not find a way to theoretically confirm \eqref{eq0301_5}. 

The main obstacle is that integrals 
\begin{equation} \label{eq_01}
\aligned 
\int_{\R} & F_{\mu_1}(z)  \overline{F_{\mu_2}(z)} \, dz \\
& =
\int_{\Re(s)=1/2}
\frac{\xi(s)}{(s-\rho_1)(\xi(s)+\xi'(s))}\cdot 
\frac{\xi(1-s)}{(1-s-\overline{\rho_2})(\xi(s)-\xi'(s))} \, ds
\endaligned 
\end{equation}
for zeros $\rho_j$ of $\xi(s)$ ($j=1,2$) 
are difficult to handle 
without using the general theory of model spaces  
which can be applied under the Riemann hypothesis, 
where $\mu_j:=i(\rho_j-1/2)$ is generally not real without the Riemann hypothesis. 
For instance, when trying to calculate these integrals 
by moving the path of integration and residue calculus, 
it becomes necessary to deal with the zeros of $\xi(s)+\xi'(s)$ 
whose relation to the zeros of $\xi(s)$ is unclear. 
Furthermore, since $\xi(s)/(s-\rho_j)$ are entire functions, 
it is unclear how to extract information about the zeros of $\xi(s)$ 
in computing integrals \eqref{eq_01}.
For these two reasons, 
it is considered to be difficult to derive some results 
on the norm of $G_n(z)$ unconditionally or 
under the assumption that the Riemann hypothesis is false, 
and new and fundamental ideas are needed.
\medskip

In \cite{Bo01}, Bombieri studies 
what happens to Weil's quadratic form $W(f(x)\ast\overline{x^{-1}g(x^{-1})})$ 
if the Riemann hypothesis is false. 
By the relation among $g_n$, $G_n$, and $W$ in the next part, 
Bombieri's results may reveal the relation between $\lambda_n$ and $\Vert G_n \Vert_{L^2(\R)}^2$ 
(such as inequality) if the Riemann hypothesis is false, 
but unfortunately, there is nothing be understood at present. 
\medskip

On the other hand, equality \eqref{s105} in Theorem \ref{thm_1} provides an inefficient method 
to disprove the Riemann hypothesis. 
Since both $\lambda_n$ and $\Vert G_n \Vert_{L^2(\R)}^2$ 
are calculable in principle to arbitrary precision, 
it could be verified by a finite computation if the equality does not hold. 
(This way was noted by the referee.)
\medskip

\subsection{On an explicit relation between $g_n$ and $G_n$.} \label{section_5_2}

As we have already seen, the model space $\mathcal{K}(\Theta)$ is defined 
for $\Theta$ in \eqref{s202} assuming the Riemann hypothesis holds. 
Then, using the orthonormal basis in Proposition \ref{prop_203}, 
the transformation $\mathcal{T}$ from the space of smooth and compactly supported functions 
on $\R_{>0}$ 
to  $\mathcal{K}(\Theta)$ is defined by 
\[
\mathcal{T}(f)(z):= \sum_{\rho=1/2-i \gamma \in \mathcal{Z}} \sqrt{\pi m_\rho}\, \widehat{f}(\rho) F_\gamma(z) ,
\]
where $\widehat{f}(s):=\int_{0}^{\infty}f(x)x^{s-1}dx$ is the Mellin transform. 
Further, definition \eqref{eq0301_2} and Proposition \ref{prop_203} give 
\[
W(f(x) \ast \overline{x^{-1}f(x^{-1})}) = \frac{1}{\pi} \left\Vert \mathcal{T}(f) \right\Vert_{L^2(\R)}^2. 
\]
This transformation extends to a wider class of functions 
defined on $\R_{>0}$, but we omit it  
because it is the subject of \cite{Su23} and not of this paper. 
However, it follows directly from \eqref{eq0225_1} and 
\[
\widehat{g_n}(\rho) = 1 - \left( 1 -\frac{1}{\rho} \right)^n
\] 
that the transformation $\mathcal{T}$ is applicable to $g_n$ of \eqref{eq0225_2} 
and $G_n = \mathcal{T}(g_n)$ holds. 

\bigskip

\noindent
{\bf Acknowledgment} 
\medskip

\noindent 
The author appreciate the referee's careful reading and valuable suggestions and comments 
that resulted in large improvements in presentation of the paper and explanation for contents.  
This work was supported by JSPS KAKENHI Grant Number JP17K05163 and JP23K03050. 

%

%
\end{document}